\documentclass{amsart}

\pdfoutput=1

\usepackage{amsmath}
\usepackage{amssymb}
\usepackage{amsthm}
\usepackage{mathabx}
\usepackage{xypic}
\usepackage{graphicx}
\usepackage{color}

\usepackage{layouts}

\newtheorem{lemma}{Lemma}[section]
\newtheorem{theorem}[lemma]{Theorem}
\newtheorem{cor}[lemma]{Corollary}
\newtheorem{prop}[lemma]{Proposition}
\newtheorem{defn}[lemma]{Definition}

\newtheorem{question}[lemma]{Question}

\newcommand{\Z}{\mathbb{Z}}

\newcommand{\C}{\mathbb{C}}

\newcommand{\into}{\hookrightarrow}

\title{Configurations of spheres in $\#^n \C P^2$}

\author{William Ballinger}

\begin{document}

\maketitle

\begin{abstract}
By taking the complements of embeddings of sphere plumbings in connected sums of $\C P^2$, we construct examples of simply
connected four-manifolds with lens space boundary and $b_2 = 1$. The resulting boundaries include many lens spaces that
cannot come from integer surgery on any knot in $S^3$, so the corresponding four-manifolds cannot be built by attaching a
single two-handle to $B^4$. Using similar constructions, we give an example of an embedded sphere in $\#^4 \C P^2$ with self-intersection number $20$, and conjecture that this is the maximum possible.  
\end{abstract}

\section{Introduction}

Every lens space $L(p,q)$ bounds a positive definite, simply connected four-manifold coming from a plumbing of spheres according to the continued fraction of $p/q$. In general, this plumbing can have quite large $b_2$, but in some cases one can make the bounding manifold much smaller. For example, if $L(p,q) = S^3_p(K)$ for some knot $K$, then the trace of this surgery will have $b_2 = 1$, while the corresponding plumbing can be larger. In general, one can ask

\begin{question}\label{motivating}
Among all positive definite simply-connected four-manifolds $X$ with $\partial X = L(p,q)$, how small can $b_2(X)$ be?
\end{question}

The answer to this question is not known for most $(p,q)$, and almost all obstructions to reducing $b_2(X)$ come from applying Donaldson's diagonalization theorem to the closed, negative definite manifold $\overline{X} \cup_{L(p,q)} Y(p,q)$, where $Y(p,q)$ is the negative definite plumbing with boundary $L(p,q)$. 

Using this technique, Lisca \cite{lisca2007lens} was able to classify those $L(p,q)$ that bound rational homology balls, or equivalently those for which the answer to Question~\ref{motivating} is $0$. A key simplification available in this case, but not for larger $b_2$, is that $L(p,q)$ bounds a rational homology ball if and only if $\overline{L(p,q)} = L(p,-q)$ does, which gives two independent lattice embedding obstructions for each lens space.

As soon as $n > 0$, however, a complete classification of lens spaces bounding positive definite four-manifolds with $b_2 = n$ is not available. For example, if there were a sphere in $\C P^2 \# \C P^2$ in the homology class $3e_1 + 2e_2$, then its complement would be positive definite, have $b_2 = 1$, and have boundary the lens space $L(13,1)$. Such a sphere seems unlikely to exist, but existing four-manifold invariants are unable to say that it does not. 

If one further restricts the four-manifold $X$, slightly more information becomes available. Using the obstruction from Donaldson's diagonaliztion theorem together with the $d$ invariant from Heegaard Floer homology, Greene \cite{greene2013lens} classified those $L(p,q)$ that are homeomorphic to $S^3_p(K)$ for some knot $K$ and $p > 0$. Equivalently, this determines those lens spaces bounding positive definite four manifolds with $b_2 = 1$ that can be built using only $2$-handles. Greene also explicity asks whether these are all of the lens spaces bounding simply connected four-manifolds with $b_2 = 1$. 

The aim of this paper is to reverse these obstructions. The closed, negative-definite four manifold $\overline{X} \cup_{L(p,q)} Y(p,q)$ appearing above is homotopy equivalent to $\#^N \overline{\C P^2}$ for some $N$, but these are not obviously diffeomorphic in general. However, in particular examples these are very often diffeomorphic (these could in fact be always diffeomorphic, since it is not currently known if there are any exotic homotopy $\#^N \overline{\C P^2}$s). If $\overline{X} \cup_{L(p,q)} Y(p,q) \cong \#^N \overline{\C P^2}$, then this diffeomorphism gives an embedding of $Y(p,q)$ in $\#^N \overline{\C P^2}$, and $X$ can be recovered as the mirror if its complement. 

In particular, we can find new examples of $X$ with small $b_2$ by first finding embeddings of the sphere plumbings $\overline{Y(p,q)}$ in the positive definite four manifold $\#^N \C P^2$, for $N$ not much larger than $b_2(Y(p,q))$. Then the complement $X = \#^n \C P^2 - \overline{Y(p,q)}$ will be a positive definite four manifold with small $b_2(X)$ and lens space boundary. The two main results obtained are:

\begin{theorem}
Infinitely many lens spaces, in particular $$L(16 n^{2} k - 16 n k^{2} - 12 n^{2} + 4 k^{2} + 8 n - 2, 16 n k - 16 k^{2} - 12 n + 4 k + 5)$$ when $1 < k < n$ and $(2k-1,2n-1) = 1$, bound simply connected, positive definite four-manifolds with $b_2 = 1$ that are not the trace of any integral surgery. 
\end{theorem}
This family should be taken as just a few of the lens spaces to which this theorem applies, since the technique used to produce these examples is much more general.

\begin{theorem}
In $\#^k \C P^2$, there is a sphere of self-intersection $5k-1$, and a sphere of self-intersection $5k$ if $k$ is divisible by $4$. Therefore, the lens space $L(5n-1,1)$ bounds a simply-connected (as long as $n \ge 3$) positive-definite four manifold with $b_2 = 4n-1$, and $L(20n,1)$ bounds a simply-connected positive-definite four manifold with $b_2 = n-1$. 
\end{theorem}

\section{The central configuration}

Beginning with three complex lines in $\C P^2$, one can repeatedly resolve intersections between them via connected sum with additional copies of $\C P^2$, resulting in embeddings in $\#^n \C P^2$ of ring-shaped plumbings of $n+2$ spheres. The self-intersection numbers of the spheres in these configurations, and other information about their homology classes, can be efficiently represented in terms of paths in the Farey graph of rational numbers. 

In what follows, a slope will be either a nonnegative rational number or $1/0$, and, given two slopes $p/q$ and $r/s$, the distance between them is $\Delta(p/q,r/s) = |ps-qr|$. If $p/q < r/s$ and $\Delta(p/q,r/s) = 1$, write $p/q \prec r/s$. 

\begin{theorem}\label{rings}

Given a sequence of slopes $$0/1 = p_0/q_0 \prec p_1/q_1 \prec \cdots \prec p_n/q_n \prec p_{n+1}/q_{n+1} = 1/0, $$ there is a configuration of spheres $\Sigma_0,\dots,\Sigma_{n+1}$ in $\#^n \C P^2$ such that $\Sigma_i \cdot \Sigma_i = \Delta(p_{i-1}/q_{i-1}, p_{i+1}/q_{i+1})$ with indices taken modulo $n+2$, and with $\Sigma_i$ intersecting $\Sigma_{i+1}$ once negatively, $\Sigma_0$ intersecting $\Sigma_{n+1}$ once positively, and all other pairs of spheres disjoint.

The complement of a tubular neighborhood of $\cup_i \Sigma_i$ is diffeomorphic to $T^2 \times D^2$, and the homology classes $\gamma_i$ represented by meridians of $\Sigma_i$ in this complement satisfy $$\gamma_i = q_i \gamma_0 + p_i \gamma_{n+1}$$.

\end{theorem}
\begin{proof}
If $n = 1$, then $p_1/q_1 = 1/1$, so taking $\Sigma_0$ and $\Sigma_2$ to be complex lines and $\Sigma_1$ to be a complex line with the orientation reversed satisfies the theorem. For $n > 1$, there will always be some $i$ with $p_{i-1}/q_{i-1} \prec p_{i+1}/q_{i+1}$, so the desired configuration can be constructed by first building the one corresponding to $$0/1 \prec \cdots \prec p_{i-1}/q_{i-1} \prec p_{i+1}/q_{i+1} \prec \cdots \prec 1/0$$ and then blowing up with another copy of $\C P^2$. 
\end{proof}

In the above construction, each newly-added $\C P^2$ creates some $\Sigma_i$, so we may give $H_2(\#^n \C P^2)$ an orthornormal basis $e_1,\dots,e_n$ with the property that $e_i$ is the class of a complex line in the copy of $\C P^2$ added to create $\Sigma_i$ (or if $p_i/q_i = 1/1$, then $e_i$ is the original $\C P^2$). 

For any slope $p/q$ not equal to $1/0$ or $0/1$, there are unique slopes $r_1/s_1$ and $r_2/s_2$ such that $r_1/s_1 \prec p/q \prec r_2/s_2$ and $r_1/s_1 \prec r_2/s_2$. We will call these the left and right parents of $p/q$, and write $r_1/s_1 \searrow p/q \nearrow r_2/s_2$. By the construction in Theorem~\ref{rings}, the left and right parents of $p_i/q_i$ will be among the $p_j/q_j$. 

\begin{lemma}
$e_i \cdot [\Sigma_j]$ is equal to $-1$ if $i = j$, to $1$ if $p_j/q_j \searrow p_i/q_i$ or $p_i/q_i \nearrow p_j/q_j$, and $0$ otherwise. 
\end{lemma}

Define vectors $w_i \in H_2(\#^n \C P^2)$ for $0 \le i \le n+1$ inductively by $w_0 = w_{n+1} = 0$, and
\begin{equation*}
w_i = w_{j_1} + w_{j_2} + e_i
\end{equation*}
if $p_{j_1}/q_{j_1} \searrow p_i/q_i \nearrow p_{j_2}/q_{j_2}$.

\begin{cor}
$w_i \cdot [\Sigma_j]$ is equal to $-1$ if $i = j$, to $q_i$ if $j = 0$, to $p_i$ if $j = n+1$, and to zero otherwise. 
\end{cor}
The nonzero components of $w_i$, sorted in increasing order, form the weight expansion $w_{p_i/q_i}$ from \cite{greene2013lens}, and their appearance here at least partly explains Greene's observation that the complementary vectors occuring in that classification are close to being weight expansions. 

\section{Configurations of many spheres: four-manifolds with lens space boundary and $b_2 = 1$}

Suppose one deletes one of the spheres from one of the configurations from Theorem~\ref{rings}, and then smooths two of the remaining intersection points. The result will be a linear plumbing of $n-1$ spheres embedded in $\#^n \C P^2$, so the complement of a neighborhood of this new configuration will have $b_2 = 1$ and lens space boundary. 

Many examples of such manifolds come from the traces of lens space surgeries, and the resulting lens spaces were classified by Greene~\cite{greene2013lens}, who also asked if any other lens space could be the boundary of a simply-connected four-manifold with $b_2 = 1$. Using this construction, we will show that the answer to this question is yes. 

Starting from one of the configurations $\Sigma_0,\dots,\Sigma_{n+1}$ from Theorem~\ref{rings}, delete $\Sigma_0$ and smooth the intersections between $\Sigma_i$ and $\Sigma_{i+1}$ and between $\Sigma_j$ and $\Sigma_{j+1}$, possibly reorienting the spheres so that the intersection signs are $\epsilon_i$ and $\epsilon_j$, respectively. 

\begin{prop}
If $p_i + \epsilon_i p_{i+1}$ and $p_j + \epsilon_j p_{j+1}$ are coprime, the complement of this new configuration is simply connected, and the image of its second homology in $H_2(\#^n \C P^2)$ is generated by $$\sigma = (p_j + \epsilon_j p_{j+1})(w_i + \epsilon_i w_{i+1} - (p_i + \epsilon_i p_{i+1})(w_j + \epsilon_j w_{j+1})$$
\end{prop}

Central to the classification from \cite{greene2013lens} is the concept of a changemaker vector:
\begin{defn}
$\sigma = (\sigma_1,\dots,\sigma_n) \in \Z^n$ is a changemaker if, after replacing its entries with their absolute values and arranging them in increasing order, $$\sigma_i \le 1 + \sum_{j < i} \sigma_j$$
\end{defn}
A slight generalization of Greene's arguments (specifically the proof of Theorem 3.3 in \cite{greene2015space}) gives the following:
\begin{prop}
If the vector $(\sigma\cdot e_i)_{i = 1}^n$ with $\sigma$ as above is not a changemaker, then the complement of the configuration is not the trace of any integral surgery. 
\end{prop}

For a particular family of examples, put $p_k/q_k = k/1$ for $0 \le k \le n$, delete $\Sigma_0$, orient the spheres so that all intersections are positive, and then smooth the intersections between $\Sigma_{n-1}$ and $\Sigma_{n}$ and between $\Sigma_{k-1}$ and $\Sigma_{k}$ for some $1 < k < n$. In this case, we have $w_i = e_1 + \cdots + e_i$ for $1 \le i \le n$, so
\begin{align*}
\sigma &= (2n-1)(w_{k-1} + w_k) - (2k-1)(w_{n-1} + w_n) \\
&= 4(n-k)(e_1+\cdots+e_{k-1}) + (2n-4k+1)e_k - (4k-2)(e_{k+1} + \cdots + e_{n-1}) - (2k-1)e_n
\end{align*}
All entries of $\sigma$ except for possibly $2n-4k+1$ are at least $3$ in absolute value, so this is never a changemaker. The self-intersection numbers of the spheres in the new configuration are $$\underbrace{2,\dots,2}_{k-2},6,\underbrace{2,\dots,2}_{n-k-2},5,n,$$ so by computing this continued fraction we get
\begin{theorem}\label{lensbound}
If $1 < k < n$ and $(2k-1,2n-1) = 1$, then the lens space $$L(16 n^{2} k - 16 n k^{2} - 12 n^{2} + 4 k^{2} + 8 n - 2, 16 n k - 16 k^{2} - 12 n + 4 k + 5)$$ bounds a simply connected, positive definite four-manifold with $b_2 = 1$ that is not the trace of any integral surgery. 
\end{theorem}
In fact, examining the classification from \cite{greene2013lens}, these lens spaces are never the result of any positive integer surgery (this is easiest to check by showing that the relevant lattices never embed as the complement of a changemaker vector).

By changing the sequence $p_i/q_i$ and the choices of which intersection points to smooth, one obtains many other values for $(p,q)$ and the complementary vector $\sigma$, and in the majority of cases $\sigma$ will not be a changemaker vector. Furthermore, in addition to smoothing two intersection points, there are a variety of other cut-and-paste techniques that can be used to eliminate two spheres from the configuration. For example, if $\Sigma_{i-1}$, $\Sigma_i$, and $\Sigma_{i+1}$ are three adjacent spheres in the configuration, one can form the connected sum of $\Sigma_{i-1}$ to $\Sigma_{i+1}$ via a tube running around $\Sigma_i$, making $\Sigma_i$ disjoint from the rest of the configuration, and then replace any of the other spheres in the configuration with its connected sum with $\Sigma_i$. These result in even more examples of lens spaces bounding positive-definite four manifolds with $b_2 = 1$, but enumerating which $L(p,q)$ occur in these constructions is not completely straightforward. 

\section{Configurations of few spheres: spheres with large square in $\#^n \C P^2$}

\begin{prop}
If $\Sigma_0,\dots,\Sigma_{n+1}$ are as in Theorem~\ref{rings}, then $$\sum_{i=0}^{n+1} \Sigma_i \cdot \Sigma_i = 3n.$$
\end{prop}
\begin{proof}
This is true for the initial configuration of three complex lines in $\C P^2$, and each blow-up introduces a new sphere of self-intersection $1$ and increases the self-intersections of two other spheres by $1$. 
\end{proof}

Therefore, if we remove from the configuration some $\Sigma_i$ with $\Sigma_i \cdot \Sigma_i = 1$, reorient the remaining $n+1$ spheres so that all of the intersections are positive, and then smooth all intersection points to produce one sphere $\Sigma$, then $\Sigma \cdot \Sigma = 5n-1$. When $n$ is not divisible by four, this seems to be the largest self intersection among spheres in $\#^n \C P^2$. However, when $n=4$, one can do slightly better:

\begin{theorem}\label{twenty}
There is a sphere of self-intersection $20$ in $\#^4 \C P^2$. 
\end{theorem}
\begin{proof}
Let $K_n$ be the twist knot with a negative clasp and $n$ half twists. Inside of a twice-punctured $\C P^2$, $K_n$ is concordant to $K_{n-1}$ (see Figure~\ref{concordance}), and the homology class of this annulus is a generator of $H_2(\C P^2)$ if $n$ is even and three times a generator if $n$ is odd. Therefore, $K_4$ is concordant to $K_0$ in a twice-punctured $\#^4 \C P^2$ via an annulus of self-intersection $1 + 9 + 1 + 9 = 20$, and since $K_4$ is slice and $K_0$ is the unknot this annulus may be capped off to the desired sphere.
\end{proof}
Taking connected sums, there is in fact a sphere of self-intersection $5n$ in $\#^n \C P^2$ whenever $n$ is divisible by $4$. 

\begin{figure}
\centering
\includegraphics[width = 0.9\textwidth]{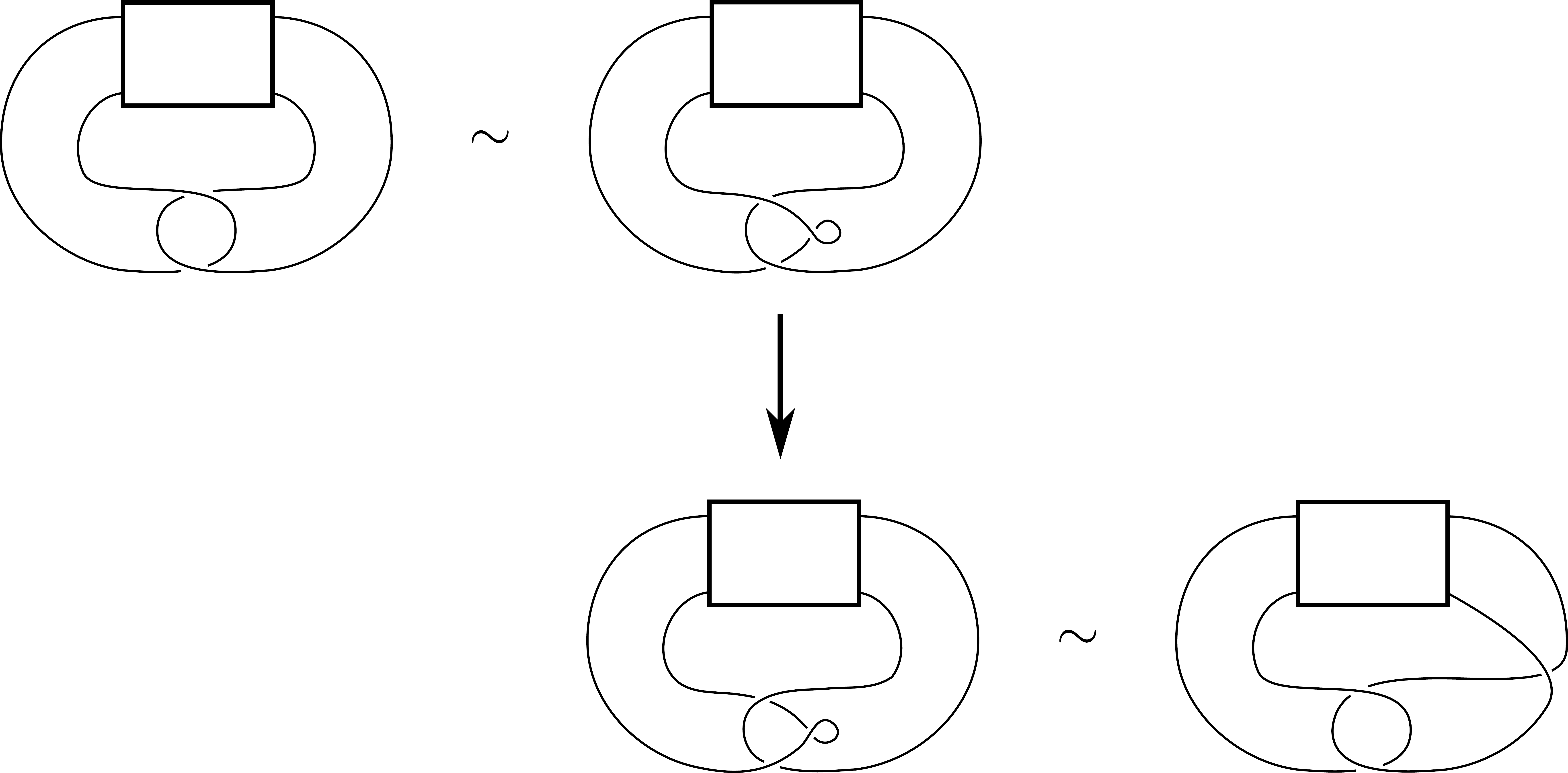}
\caption{A full positive twist transforms $K_n$ into $K_{n-1}$, with linking number $1$ if $n$ is even and $3$ if $n$ is odd.}\label{concordance}
\end{figure}

This construction can be related to the previous ones by adding a disk whose boundary goes once around the ``ring" to the configuration of spheres $\Sigma_0,\dots,\Sigma_{n+1}$; a neighborhood of the resulting complex is a four manifold built by attaching $n+1$ two-handles to $B^4$ along a circular chain of unknots. By sliding one of these two handles over all of the others, one gets an embedding of the trace of $5n$ surgery on the twist knot $K_n$ in $\#^n \C P^2$, and when $n = 4$ this trace contains a sphere. 

While there does not seem to be a way to construct a sphere of self intersection larger than $5n$ in $\#^n \C P^2$, existing invariants are not able to prove that none exist. However, the sphere $\Sigma$ constructed in Theorem~\ref{twenty} is at least locally optimal - if two homologically essential spheres in $\#^n \C P^2$ intersect in at most one point, then their connected sum will be an embedded sphere of larger self-intersection number, but no such construction can improve $\Sigma$:

\begin{prop}
Any homologically essential sphere in $\#^4 \C P^2$ intersecting $\Sigma$ transversally does so in at least $3$ points.
\end{prop}
\begin{proof}
$\Sigma \cdot \Sigma = 20$, and the homology class $[\Sigma]$ is characteristic in $H_2(\#^4 \C P^2)$. Blowing up (with copies of $\overline{\C P^2}$) $19$ points on $\Sigma$ reduces the self-intersection to one and keeps the homology class characteristic, and then blowing down the resulting $+1$ sphere results in a simply connected, spin four-manifold $X$ with $b_2^+ = 3$ and $b_2^- = 19$. This new $X$ is therefore homotopy equivalent to a K3 surface, so by Theorem~1.1 from~\cite{morgan1997homotopy} it has a nonzero Seiberg-Witten invariant in the $\operatorname{Spin}^c$ structure with vanishing $c_1$, and so any homologically essential embedded sphere has negative self-intersection. 

However, if $\Sigma'$ is homologically essential and intersects $\Sigma$ transversally in at most two points, then looking at possible homology classes of $\Sigma'$ shows that $\Sigma' \cdot \Sigma' \ge |\Sigma' \cap \Sigma|$, so if some of the blow-ups used to construct $X$ are done at these intersection points, the result will be a homologically essential sphere in $X$ with nonnegative self intersection, which cannot exist. 
\end{proof}

One could ask whether the homotopy K3 surface in this proof is diffeomorphic to the K3 surface. The following result suggests that it might be:
\begin{prop}\label{S2K3origin}
There exist a positive definite four-manifold $X$ with $b_2 = 4$ and a sphere $\Sigma$ embedded in $X$ with self-intersection $20$ such that blowing up $19$ points on the sphere and then blowing down results in the K3 surface.
\end{prop}
\begin{proof}
It will be simpler to work backwards, beginning with $\operatorname{K3} \# \C P^2$ and finding $19$ spheres of self-intersection $-1$ that each meet a complex line of $\C P^2$ in one point. In \cite{finashin199786}, Finashin and Mikhalkin construct a particular configuration of $25$ spheres of self-intersection $-2$, whose intersection graph is a copy of the Petersen graph with an extra vertex inserted into each edge. This graph has an induced path of length $19$, so $19$ of these two-spheres are arranged in a path such that each one intersects only it's neighbors transversally once. Sliding these spheres over the complex line in $\C P^2$ produces the desired spheres of self-intersection $-1$. 
\end{proof}
Whether this $X$ is diffeomorphic to $\#^4 \C P^2$, and whether this $\Sigma$ is the sphere described earlier, seem difficult to answer. 

Since the complement of the sphere of Proposition~\ref{S2K3origin} embeds in the K3 surface, a symplectic four-manifold, results of Hom and Lidman \cite{hom2018note} (originally applied to for closed manifolds, but also applicable here) imply that, for this $X$ and $\Sigma$, the complement $X - \Sigma$ requires one-handles in any handle decomposition:
\begin{proof}
Since $X - \Sigma$ embeds in a closed symplectic four-manfold with $b_2^+ > 1$ (the K3 surface), Proposition~2.7 in \cite{hom2018note} implies that if $K$ is any knot such that the trace of $n$-surgery on $K$ embeds in $X - \Sigma$, then $V_0(K) > 0$. However, using the fact that $X$ is positive definite and adapting the proof of Theorem~1.1 from \cite{hom2018note}, we will be able to show that if $X - \Sigma$ had a handle decomposition with no one-handles, then there would be a knot $K$ whose trace embeds in $X - \Sigma$ with $V_0(K) = 0$. 

Any handle decompositiong of $X - \Sigma$ with only two- and three-handles gives a generating set for $H_2(X-\Sigma)$ coming from capped-off Seifert surfaces of the knots representing the two-handles, and introducing a cancelling two- and three- handle pair followed by a series of handleslides can include any class $\alpha \in H_2(X-\Sigma)$ in this generating set. This means that, for any such $\alpha$, there is an embedding $W_n(K) \into X - \Sigma$ of the $n$-trace of some knot $K \subset S^3$ taking a generator of the homology of $W_n(K)$ to $\alpha$. In particular, with respect to an orthonormal basis of $H_2(X)$ in which $\Sigma = (1,1,3,3)$, we can take $ \alpha = (1,-1,0,0)$. Working in the same basis, let $\mathfrak{s}$ be a $\operatorname{Spin}^c$ structure on $X$ with $c_1(\mathfrak{s}) = (1,-1,1,1)$. By Proposition~1.6 from \cite{ni2015cosmetic} $d(S^3_2(K), \mathfrak{s}|_{S^3_2(K)}) = d(L(2,1),0) - 2V_0(K)$. However, since the complement $X - W_2(K)$ is positive-definite, we also have the bound $d(S^3_2(K), \mathfrak{s}|_{S^3_2(K)}) \ge d(L(2,1),0)$, which forces $V_0(K) = 0$. 
\end{proof}
However, while the complement of $\Sigma$ has vanishing first homology, it is not clear that it is simply connected.

\bibliographystyle{amsplain}
\bibliography{CP2spheres.bib}

\providecommand{\bysame}{\leavevmode\hbox to3em{\hrulefill}\thinspace}
\providecommand{\MR}{\relax\ifhmode\unskip\space\fi MR }
\providecommand{\MRhref}[2]{%
  \href{http://www.ams.org/mathscinet-getitem?mr=#1}{#2}
}
\providecommand{\href}[2]{#2}
\begin{thebibliography}{1}

\bibitem{finashin199786}
Sergey Finashin and Grigory Mikhalkin, \emph{A (--86)-sphere in the {K3}
  surface}, Turkish Journal of Mathematics \textbf{21} (1997), no.~1, 129--131.

\bibitem{greene2013lens}
Joshua~Evan Greene, \emph{The lens space realization problem}, Annals of
  Mathematics (2013), 449--511.

\bibitem{greene2015space}
\bysame, \emph{L-space surgeries, genus bounds, and the cabling conjecture},
  Journal of Differential Geometry \textbf{100} (2015), no.~3, 491--506.

\bibitem{hom2018note}
Jennifer Hom and Tye Lidman, \emph{A note on positive-definite, symplectic
  four-manifolds}, Journal of the European Mathematical Society \textbf{21}
  (2018), no.~1, 257--270.

\bibitem{lisca2007lens}
Paolo Lisca, \emph{Lens spaces, rational balls and the ribbon conjecture},
  Geometry \& Topology \textbf{11} (2007), no.~1, 429--472.

\bibitem{morgan1997homotopy}
John~W Morgan and Zolt{\'a}n Szab{\'o}, \emph{Homotopy {K3} surfaces and mod 2
  {Seiberg-Witten} invariants}, Mathematical Research Letters \textbf{4}
  (1997), no.~1, 17--21.

\bibitem{ni2015cosmetic}
Yi~Ni and Zhongtao Wu, \emph{Cosmetic surgeries on knots in s3}, Journal
  f{\"u}r die reine und angewandte Mathematik (Crelles Journal) \textbf{2015}
  (2015), no.~706, 1--17.

\end{thebibliography}

\end{document}